\documentclass[12pt,psamsfonts]{amsart}
\usepackage{amsmath}
\usepackage{amsthm}
\usepackage{amssymb}
\usepackage{amscd}
\usepackage{amsfonts}
\usepackage{amsbsy}
\usepackage{graphicx}
\usepackage[dvips]{psfrag}

\newtheorem {theorem*}{Theorem}
\newtheorem {theorem} {Theorem}

\newtheorem {lemma}  [theorem]{Lemma}

\newtheorem {remark} [theorem]{Remark}

\begin{document}
\title[Limit cycle  for semi-quasi-homogeneous polynomial vector fields] {\bf Limit cycles for planar
semi-quasi-homogeneous polynomial vector fields}

\author[ Yulin Zhao]
{ Yulin Zhao}
\address{Yulin Zhao, Department of Mathematics, School of  Mathematics and Computational Science,  Sun Yat-sen University, Guangzhou, 510275, People's Republic of  China.}
\email{mcszyl@mail.sysu.edu.cn,ylinzhao@yahoo.com.cn}

\thanks{ 
Supported by the NSF of China (No.11171355), the Ph.D. Programs Foundation of Ministry of Education of China (No. 20100171110040)  and Program for New Century Excellent Talents in University.}

\date{}
\dedicatory{}

\maketitle

\maketitle

\begin{abstract}
This paper is concerned with the limit cycles for planar semi-quasi-homogeneous polynomial systems.  We give some explicit  criteria for the nonexistence and existence of periodic orbits. Let $N=N(p,q,m,n)$  be the maximum number of limit cycles of such system. A lower bound is given for $N=N(p,q,m,n)$.   The cyclicity and center problem are studied for some subfamilies of semi-quasi-homogeneous polynomial systems. Our results generalize those obtained for polynomial semi-homogeneous systems.
\end{abstract}

\section{Introduction and statement of the main results}\label{s1}

A function $f(x,y)$ is called a {\it $(p,q)$-quasi-homogeneous function of weighted degree $m$} if $f(\lambda^p
x,\lambda^q y)=\lambda^m f(x,y)$ for all $\lambda\in \mathbb{R}$. Let $P_m(x,y)$ and $Q_n(x,y)$ be $(p,q)$-
quasi-homogeneous polynomials of weighted degree $p-1+m$ and $q-1+n$, respectively.  We say that $X=(P_m(x,y),Q_n(x,y))$ is a {\it planar $(p,q)$ semi-quasi-homogeneous polynomial
vector field} if  $m\not=n$. The system of differential equations associated to $X$ is
\begin{equation}\label{eq1}
\frac{dx}{dt}=P_m(x,y)=\sum_{pi+qj=p+m-1}a_{ij}x^iy^j,\,\,\,\,\,\frac{dy}{dt}=Q_n(x,y)=\sum_{pi+qj=q+n-1}b_{ij}x^iy^j.
\end{equation}
Here $p, q,\,m$ are positive integers and $P_m(x,y)$ and $Q_n(x,y)$ are coprime in the ring $\mathbb{R}[x,y]$. To be short we denote this by  $(P(x,y),Q(x,y))=1$.

Observed that the above definition is the natural one for the following reason:
\begin{itemize}
\item[(i)] When $p=q=1$, $X$ is called {\it a semi-homogeneous vector field}\cite{CL1,CL2}. That is to say, the above definition  coincides with the usual definition of the homogeneous cases.
\item[(ii)] If $m=n$, then  $X$ is called {\it a $(p,q)$-quasi-homogeneous vector field of weighted degree $m$}, see \cite{AA}, Chapter 7. In particular if $p=q=1$, $X$ is {\it a homogeneous polynomial vector field of degree $m$}.
\item[(iii)] The finite (resp. infinity) singular points of semi-homogeneous vector field can be studied by homogeneous blow-up ({\it resp.} Poincar\'e compactification)  whereas the one of semi-quasi-homogeneous vector fields can  be studied by the use of quasi-homogeneous blow-up ({\it resp.} Poincar\'e-Lyapunov compactification)\cite{DLA}.
\end{itemize}

One of the classical problem in the qualitative theory of planar polynomial  systems is the study of the number of limit cycles, which is known as the 16th Hilbert's problem. There has been a substantial amount of work devoted to solving this problem for homogeneous, quasi-homogeneous or semi-homogeneous polynomial vector fields. W. Li {\it etc.}\cite{LLYZ} provided an upper bound for the maximum number of limit cycles bifurcating from the period annulus of any homogeneous and quasi-homogeneous centers, which can be obtained using the Abelian integral method. Gavrilov {\it etc.}\cite{GGG} gave a more general results on the limit cycles bifurcated from the periodic orbits of quasi-homogeneous centers.  Cima and Llibre \cite{CL} investigated the algebraic and topological classification of  homogeneous cubic vector fields. In 1997, Cima, Gasull and Manosas\cite{CGM2}  studied the limit cycles for semi-homogeneous vector fields. They proved that semi-homogeneous  systems can exhibit periodic orbits only when $nm$ is odd and there exists such system with at least $(n+m)/2$ limit cycles.  In the papers \cite{CL1,CL2} Cair\'o and Llibre study the phase portraits of semi-homogeneous vector fields with $(m,n)=(1,2)$ and $(m,n)=(1,3)$ respectively. Some papers  are concerned with the  integrability \cite{chavaga,LZ} and structural stability\cite{LRR,LRR1,OZ} for quasi-homogeneous or semi-homogeneous  polynomial vector fields.

It is always be possible to decompose a vector field $\mathcal {X}$ in a formal series centered in $0\in \mathbb{R}^n$: $\mathcal{X}=\sum_{j\geq k}X_j$, where $X_j$ are $(p,q)$-quasi-homogeneous vector field of weighted degree $k$, $k$ is the first integer such that $X_k\not=0$\cite{BM}. Therefore  in order to get more information on limit cycles of general vector fields $\mathcal{X}$ it is essential to study the geometrical properties of  quasi-homogeneous and semi-qusi-homogeneous vector field. We note that Coll, Gasull and Prohens have investigated planar vector fields defined by the sum of two quasi-homogeneous vector fields\cite{CGP}.

In this paper we  study the limit cycles of planar semi-quasi-homogeneous polynomial vector fields, defined in \eqref{eq1}. To simplify the statements of the main results, first of all we give the following lemma, which will be proved in section \ref{s2}.

\begin{lemma}\label{l1}
Suppose $(p,q) = k \geq 2$ in \eqref{eq1}, then there exists a unique vector  $(p',q',m',n')$ with $(p',q')=1$ such that  system \eqref{eq1} is a
$(p',q')$ semi-quasi-homogeneous vector field.
\end{lemma}

By Lemma \ref{l1},  we   suppose that the following conventions hold without loss of generality.

\medskip

\noindent{\bf Convention} {\it  We  always  assume in this paper  that
\begin{itemize}
  \item[(i)] $p$ is odd and $(p,q)=1$, and
  \item [(ii)] $P_m(x,y)$ and $Q_m(x,y)$ are coprime polynomials.
\end{itemize}}
\medskip

Our main results are the following four  theorems. The first one provides  explicit  criteria for the nonexistence and existence of periodic orbits, and a lower bound of $N=N(p,q,m,n)$ for  the maximum number of limit cycles of system \eqref{eq1}.
\begin{theorem}\label{mainth}
Let $P_m(x,y)$ and $Q_n(x,y)$ are $(p,q)$-
quasi-homogeneous polynomials of degree $p-1+m$ and $q-1+n$, respectively, $m\not=n$. For system \eqref{eq1}, the following statements hold:
\begin{itemize}
\item[(i)] If either $q\nmid p+m-1$ or $p\nmid q+n-1$, then system \eqref{eq1} has no periodic orbit.
\item[(ii)] If $q| p+m-1$ and  $p| q+n-1$, then system \eqref{eq1} is reduced to the following normal form
\begin{equation}\label{eq1'}
\frac{dx}{dt}=P_m(x,y)=\sum_{i=0}^{[r_1/p]}a_{i}x^{iq}y^{r_1-ip},\,\,\,\,\,\frac{dy}{dt}=Q_n(x,y)=\sum_{j=0}^{[r_2/q]}b_jx^{r_2-jq}y^{jp}.
\end{equation}
where $r_1=(p+m-1)/q,\,\,r_2=(q+n-1)/p$, $[r_1/p]$ denotes the integer part of $r_1/p$.

 If system \eqref{eq1'}  has  periodic orbit, then one of the following conditions holds:
\begin{itemize}
   \item [(ii.1)] If both $p$ and $q$ are odd, then $m,\,n,\,r_1$ and $r_2$ are odd.
   \item [(ii.2)] If $p$ is odd and  $q$ is even, then $m$ and $n$ are  even,  $r_1$ and $r_2$  are  odd, respectively.
 \end{itemize}

\item[(iii)] If (ii.1) holds, then there exist $P_m(x,y)$ and $Q_n(x,y)$ such that each of the following assertions holds:
\begin{itemize}
\item[(a)] All solutions of \eqref{eq1'} are periodic.
\item[(b)] No solutions of \eqref{eq1'} are periodic.
\item[(c)] There are periodic and non-periodic solutions of system \eqref{eq1'}.
\end{itemize}
Let  $N=N(p,q,m,n)$  be the maximum number of limit cycles of system \eqref{eq1'}.  Then $N\geq [([r_1/p]+1)/2]+[([r_2/q]+1)/2]-1$.
\item[(iv)] If (ii.2) holds and system \eqref{eq1'} has periodic orbit, then the origin is a center.
\end{itemize}
\end{theorem}
The results in Theorem \ref{mainth} generalize those obtained for semi-homogeneous polynomial vector fields, proved in \cite{CGM2}.

One of important problem in the qualitative theory of planar polynomial vector fields is the study of the local phase portrait at the singularities to characterize when a singular is a center, which is called the {center problem}. A singular point $O$ of a real planar analytic vector field $\mathcal{X}$ is called a
{\it center} if there exists a neighborhood $U$ of $O$ such that $U\backslash\{O\}$ is filled with periodic integral curves of the vector field. The center is called {\it elementary} if the
linearization $d\mathcal{X}(O)$ of the vector field is a rotation, otherwise the center is called non-elementary. If the origin is a center of system \eqref{eq1'}, then it is in general non-elementary.  The {\it cyclicity}  is  the total number of limit cycles which can emerge from a configuration of trajectories (center point, period annulus, separatrix cycle) under a perturbation.

The center and its cyclicity problem are only solved for some special kind of planar polynomial systems: quadratic systems, cubic systems with homogeneous nonlinearities, and so on. In the following two theorems, we give some more information on the center and cyclicity problem for some subfamily of the semi-quasi-homogeneous system \eqref{eq1'}.

\begin{theorem}\label{th3} Suppose that $q| p+m-1,\,p| q+n-1$ and the condition (ii.1) in Theorem \ref{mainth} holds.  Consider the  system
\begin{equation}\label{eq3}
\frac{dx}{dt}=\sum_{i=0}^{[r_1/p]}a_{i}x^{iq}y^{r_1-ip},\,\,\,\,\,\frac{dy}{dt}=b_0x^{r_2}
\end{equation}
with $r_1>r_2$ and $m>n$. If $a_0b_0<0$, then the following statements hold.
\begin{itemize}
\item[(i)] The origin is a center if and only if $a_1=a_3=\cdots=a_\kappa=0$, where
\begin{equation*}
\kappa=\left\{\begin{array}{ll}\left[\dfrac{r_1}{p}\right],&\mathrm{if}\,\,\,\left[\dfrac{r_1}{p}\right]\,\,\mathrm{is\,\,odd,}\\[2ex]
\left[\dfrac{r_1}{p}\right]-1,&\mathrm{if}\,\,\,\left[\dfrac{r_1}{p}\right]\,\,\mathrm{is\,\,even.}\end{array}\right.
\end{equation*}
\item[(ii)] If $a_1=a_3=\cdots=a_{2i-3}=0$, $a_{2i-1}\not=0,\,\,2i-1\leq \kappa$, then the origin is a weak focus and under perturbations inside the family it can produce $i-1$ limit cycles.
\item[(iii)] The cyclicity of the center is $(\kappa-1)/2=[([r_1/p]-1)/2]$.
\end{itemize}
\end{theorem}

In the next theorem we study the local phase portraits at the origin for  system \eqref{eq1'} with $m=1$.
 \begin{theorem}\label{th4}
 Suppose that $m=1$ in \eqref{eq1'} and  the condition (ii.1) in Theorem \ref{mainth} holds. Then system \eqref{eq1'} is reduced to
 \begin{equation}\label{eq4}
\frac{dx}{dt}=P_1(x,y)=a_1x+a_0y^p,\,\,\frac{dy}{dt}=Q_n(x,y)=\sum_{j=0}^{n/p}b_jx^{n/p-j}y^{jp}
\end{equation}
where $n/p$ is a odd number and $(P_1(x,y),Q_n(x,y))=1$.
\begin{itemize}
\item[(i)] Suppose that $a_0=0,\,\,a_1\not=0$, then $b_{n/p}\not=0$. Moreover  the origin is a unstable (resp. stable)  node if $a_1>0,\,\,b_{n/p}>0$  (resp.$a_1<0,\,\,b_{n/p}<0$)  and a saddle if  $a_1b_{n/p}<0$ respectively.
\item[(ii)] Suppose  $a_1\not=0,\,a_0\not=0$.
\begin{itemize}
\item[(a)] If $p=1$, then
\begin{itemize}
\item[(a.1)] If $\sum_{j=0}^n(-1)^{n-j}b_j/(a_0^{j-1}a_1^{n-j+1})>0$ and $a_1>0$ (resp. $a_1<0$), then the origin is a unstable (resp. stable)  node .
 \item[(a.2)] If $\sum_{j=0}^n(-1)^{n-j}b_j/(a_0^{j-1}a_1^{n-j+1})<0$, then the origin is a saddle.
\end{itemize}
\item[(b)] If $p\geq 3$, then
\begin{itemize}
\item[(b.1)] If $a_1(\sum_{j=0}^{n/p}b_j(-a_0/a_1)^{n/p-j})>0$ and $a_1>0$ (resp. $a_1<0$), then the origin is a unstable (resp. stable)  node.
 \item[(b.2)] If $a_1(\sum_{j=0}^{n/p}b_j(-a_0/a_1)^{n/p-j})<0$, then the origin is  a saddle.
\end{itemize}
\end{itemize}
\item[(iii)] Suppose   $a_0b_0<0,\,a_1=0,\,\,n>p^2$, then the origin is a center if and only if $b_1=b_3\cdots=b_{n/p}=0$.
\end{itemize}
 \end{theorem}

Finally, we study the cyclicity at infinity for a kind of dual version of the above vector field. Before that  we introduce the $(l_1,l_2)$-trigonometric function  which were given by Lyapunov \cite{L} in his study of the stability of degenerate singular points. Let $z(\phi)=\mathrm{Cs}\phi$ and $\omega(\phi)=\mathrm{Sn}\phi$ be the solution of the following initial problem
\begin{equation*}
\dot z=-\omega^{2l_1-1},\,\,\dot\omega=z^{2l_2-1},\,\,z(0)=l_1^{-\frac{1}{2l_2}},\,\,\omega(0)=0.
\end{equation*}
For $(l_1,l_2)=(1,1)$, we have that ${\mathrm Cs}\phi=\cos
\phi,\,\,{\mathrm Sn}\phi=\sin\phi$, i.e. the $(1,1)$-trigonometric
functions are the classical ones.

Consider the system
 \begin{equation}\label{eq5}
\frac{dx}{dt}=-y^{r_1},\,\,\frac{dy}{dt}=x^{r_2}+\sum_{j=1}^{[r_2/q]}b_jx^{r_2-jq}y^{jp},
\end{equation}
where $r_1=2l_1-1$ and $r_2=2l_2-1$ are defined in Theorem \ref{mainth}.

In order to study the behavior of the trajectories of a planar differential system near infinity it is possible to use a compactification. In this paper we prefer to reduce the ``infinity'' to a
single point. The construction of this compactification is as follows (see for instance \cite{CGM1}): we first write system \eqref{eq5} in the generalized polar coordinate $(\rho,\phi)$ by the transformation
\begin{equation}\label{pc}x=\rho^{l_1}\mathrm{Cs}\phi,\,\,\,y=\rho^{l_2}\mathrm{Sn}\phi,
\end{equation}
then change $\rho$ by $1/\rho$ and finally reparametrize the system to obtain  a smooth one.

It is said that a planar differential equation has a center at infinity if there is a compact set of $\mathbb{R}^2$ such that outside it all trajectories of that differential equations are closed\cite{CGM1}.

\begin{theorem}\label{th5}
 Suppose that  $q| p+m-1,\,p| q+n-1,\,r_1>r_2,\,m>n$ and the condition (ii.1) in Theorem \ref{mainth} holds.  For system \eqref{eq5}, the following statements hold.
\begin{itemize}
\item[(i)] A neighborhood of infinity is filled with periodic orbits   if and only if $b_1=b_3=\cdots=b_\mu=0$, where
\begin{equation*}
\mu=\left\{\begin{array}{ll}\left[\dfrac{r_2}{q}\right],&\mathrm{if}\,\,\,\left[\dfrac{r_2}{q}\right]\,\,\mathrm{is\,\,odd,}\\[2ex]
\left[\dfrac{r_2}{q}\right]-1,&\mathrm{if}\,\,\,\left[\dfrac{r_2}{q}\right]\,\,\mathrm{is\,\,even.}\end{array}\right.
\end{equation*}
\item[(ii)] If $b_1=b_3=\cdots=b_{2i-3}=0$, $b_{2i-1}\not=0,\,\,2i-1\leq \mu$, then the stability of the infinity is given by the sign of $b_{2i-1}$. Moreover the maximum number number of periodic orbits that bifurcate from infinity   inside the family  for small perturbations is  $i-1$.
\item[(iii)] The cyclicity of the infinity  is $(\mu-1)/2=[([r_2/q]-1)/2]$.
\end{itemize}
\end{theorem}

Theorem \ref{th3} and  Theorem \ref{th5} generalize those for semi-homogeneous polynomial vector fields in \cite{CGM2}.

\section{Some preliminary results }\label{s2}

In this section we give some preliminary results which will be used in the proof of the main results. First of all we prove Lemma \ref{l1}.
\begin{proof}[Proof of Lemma \ref{l1}]
Let $p=kp',\,q=kq'$. Since $(p,q) = k \geq 2$, we have  $(p',q')=1$.  It follows from the definition of  $P_m(x,y)$ in \eqref{eq1} that $k p'- 1 + m = k p' i + k q'j$, which implies  that $k|m-1$. Using the same arguments one obtains $k|n-1$. Taking $m'=1+(m-1)/k$ and $n'=1+(n-1)/k$,  the  assertion follows.
\end{proof}

\begin{lemma}\label{l6}\cite{CGM2,L} ${\mathrm Cs}\phi$ and ${\mathrm Sn}\phi$, defined in Section \ref{s1},  have the following properties.
\begin{itemize}
\item[(i)] $l_1{\mathrm Cs}^{2l_2}\phi+l_2{\mathrm Sn}^{2l_1}\phi=1$.
\item[(ii)] $d{\mathrm Sn}\phi/d\phi={\mathrm Cs}^{2l_2-1}\phi$,  $d{\mathrm Cs}\phi/d\phi=-{\mathrm Sn}^{2l_1-1}\phi$.
\item[(iii)] ${\mathrm Cs}(-\phi)={\mathrm Cs}\phi$, ${\mathrm Sn}(-\phi)=-{\mathrm Sn}\phi$.
\item[(iv)] ${\mathrm Cs}\phi$ and ${\mathrm Sn}\phi$ are
$T$-periodic functions with
\begin{equation*}
T=T_{l_1,l_2}=2l_1^{-\frac{1}{2l_2}}l_2^{-\frac{1}{2l_1}}\frac{\Gamma(\frac{1}{2l_1})\Gamma(\frac{1}{2l_2})}{\Gamma(\frac{1}{2l_1}+\frac{1}{2l_2})},
\end{equation*}
where $\Gamma$ denotes the gamma function.
\end{itemize}
\end{lemma}

\begin{lemma}\label{l7}\cite{CGM2} $\int_0^T\mathrm{Sn}^\alpha\phi\mathrm{Cs}^\beta\phi d\varphi\not=0$ if and only if $\alpha$ and $\beta$ are even.
\end{lemma}

\medskip
Finally, we give two theorems which characterize the local phase portraits at  the singular points.

\begin{theorem}\label{th8}\cite{DLA} Let $(0,0)$ be an isolated singular point of the vector field $X$ given by
\begin{equation}
\dot x=A(x,y),\,\,\,\,\dot y=\lambda y+B(x,y),
\end{equation}
where $A(x,y)$ and $B(x,y)$ are analytic in a neighborhood  of the origin with $A(0,0)=B(0,0)=DA(0,0)=DB(0,0)=0$ and $\lambda>0$. Let $y=f(x)$ be the solution of the equation
\begin{equation*}
\lambda y+B(x,y)=0
\end{equation*}
in a neighborhood of the origin, and suppose the function $A(x,f(x))$ has the expression
\begin{equation*}
A(x,f(x))=a_mx^m+o(x^m),
\end{equation*}
where $m\geq 2$ and $a_m\not=0$. Then there always exists an invariant analytic curve, called the strong unstable manifold, tangent at $(0,0)$ to the $y$-axis, on which $X$ is analytically conjugate to $\dot y=\lambda y$; it represents repelling behavior since $\lambda>0$. Moreover the following statements hold:
\begin{itemize}
\item[(i)] If $m$ is odd and $a_m<0$, then $(0,0)$ is a topological saddle.
\item[(ii)] If $m$ is odd and $a_m>0$, then $(0,0)$ is a unstable topological  node.
\item[(iii)] If $m$ is even, then $(0,0)$ is a saddle-node, that is, a singular point whose neighborhood is the union of one parabolic and two hyperbolic sectors.
\end{itemize}
\end{theorem}

We say that a family of vector fields $X_\lambda,\,\,\lambda\in \mathbb{R}^k$ is {\it stable by  homoteces} if all $\lambda\in \mathbb{R}^k$ and  for all $\epsilon\in \mathbb{R}$ there exists $\lambda'\in \mathbb{R}^k$ such that $\epsilon X_\lambda=X_{\lambda'}$.  Fixed $p,q$, any analytic vector field $X$, such that $X(0)=0$, can be decomposed as
\begin{equation*}
X=\sum_{i=1}^\infty h_i^{(p,q)}(X),
\end{equation*}
where $h_i^{(p,q)}(X)$ is $(p,q)$-quasi-homogeneous vector fields of weighted degree $i$\cite{CGM1}.
\begin{theorem}\label{th9}\cite{CGM1}Consider the vector field
\begin{equation}\label{eq7}
X=-y^n\frac{\partial}{\partial x}+x^m\frac{\partial}{\partial y}+Y(x,y,c)+\sum_{i=1}^lb_iZ_i(x,y),
\end{equation}
where $n=2p-1\geq m=2q-1$, are natural numbers, $c\in \mathbb{R}^u,\,\,u\in \mathbb{R}^\mathbb{N}$, and $b=(b_1,b_2,\cdots,b_l)\in\mathbb{R}^l,\,\,l\in\mathbb{N}$, and
\begin{itemize}
\item[(i)] $Y(x,y,c)=Y_1(x,y,c)(\partial/\partial x)+Y_2(x,y,c)(\partial/\partial y)$ has analytic components, $Y(0,0,c)=0$, $h_i^{(p,q)}(Y(x,y,c))=0$ for all $i\leq k=2pq-p-q+1$ (the weighted degree of $-y^n(\partial/\partial x)+x^m(\partial/\partial y)$) and the differential equations associated $b=0$ has a center at the origin.
\item[(ii)] $Z_i$ are $(p,q)$-quasi-homogeneous vector fields of weighted degree $d_i$, greater than $k=2pq-p-q+1$ and $d_i>d_j$ when $i>j$. Furthermore
\begin{equation*}
I_i=\iint_{px^{2q}+qy^{2p}\leq 1}\mathrm{div} Z_i(x,y)dxdy\not=0
\end{equation*}
for all $i$.
\end{itemize}
 Then for the differential equations associated to \eqref{eq7} the following hold.
\begin{itemize}
\item[(a)] The origin is a center if and only if $b=0$.
\item[(b)] If $b_1=b_2=\cdots=b_{i-1}=0$ and $b_i\not=0$ the stability of the origin is given by the sign $b_iI_i$ and its cyclicity inside the family \eqref{eq7} is $i-1$.
\item[(c)] The cyclicity of the center inside the family \eqref{eq7} is $l-1$.
\item[(d)] If in addition the family $Y(x,y,c)$ is stable by homoteces, then the maximum number of limit cycles that can bifurcate for $\epsilon$ small enough, from the curve $px^{2q}+qy^{2p}=h$ for the vector field
    \begin{equation*}
    -y^n\frac{\partial}{\partial x}+x^m\frac{\partial}{\partial y}+\epsilon\left(Y(x,y,c)+\sum_{i=1}^lb_iZ_i(x,y)\right),
    \end{equation*}
    is $l-1$. Moreover there  exist $b\in \mathbb{R}^l$ and $c\in\mathbb{R}^u$ such that this number is attained.
\end{itemize}
\end{theorem}

The following  remark was given in \cite{CGM1}.
\begin{remark}\label{r10}
If we change the hypotheses of Theorem \ref{th9} by
\begin{itemize}
\item[(i$'$)]$Y(x,y,c)$  has polynomial  components,  $h_i^{(p,q)}(Y(x,y,c))=0$ for all $i\geq k=2pq-p-q+1$  and the differential equations associated $b=0$ has a center at infinity.
\item[(ii$'$)] $Z_i$ are $(p,q)$-quasi-homogeneous vector fields of weighted degree $d_i$, $d_i<d_j$ when $i>j$, and  $d_1<k=2pq-p-q+1$, and furthermore
\begin{equation*}
I_i=\iint_{px^{2q}+qy^{2p}\leq 1}\mathrm{div} Z_i(x,y)dxdy\not=0\,\,\,\,\,\,\mathrm{for \,\,all}\,\, i,
\end{equation*}
\end{itemize}
then the same conclusions hold, changing the word ``origin" to ``infinity".
\end{remark}

The above two theorems and remark are important for the  proof of Theorem \ref{th3}-\ref{th5}.

\section{Proof of Theorem \ref{mainth}}\label{s3}

In this section we will  prove  Theorem \ref{mainth}.
\begin{lemma}\label{l10}
The origin is a unique (real and finite) singular point of system \eqref{eq1}.
\end{lemma}
\begin{proof}By definition of $(p,q)$ semi-quasi-homogeneous systems, we have that
\begin{equation}\label{eq8}
P_m(\lambda^px,\lambda^qy)=\lambda^{p+m-1}P_m(x,y),\,\,\, Q_n(\lambda^px,\lambda^qy)=\lambda^{q+n-1}Q_n(x,y).
\end{equation}
Taking $\lambda=0$ into \eqref{eq8} for fixed $(x,y)$, one obtains $P_m(0,0)=Q_n(0,0)=0$. Therefore the origin is a singular point of system \eqref{eq1}.

Let $(a,b)$ be the singular point of system \eqref{eq1} with $(a,b)\not=(0,0)$. It follows from \eqref{eq8} that $P_m(\lambda^pa,\lambda^qb)=Q_n(\lambda^pa,\lambda^qb)=0$. This yields that
the curve $\{(\lambda^pa,\lambda^qb):\lambda\in\mathbb{R}\}$ is fulfilled of singular points, in contradiction that $P_m(x,y)$ and $Q_n(x,y)$ are coprime.
\end{proof}

\begin{proof}[Proof of Theorem \ref{mainth}(i)] Suppose that $P_m(x,y)$ is defined in \eqref{eq1}. The pair $(i,j)$ in the expression of $P_m(x,y)$ is a non-negative solution of the equation
\begin{equation}\label{eq9}
pi+qj=p-1+m.
\end{equation}
As usual, we use the notation {\it non-negative integer (resp.  integer) solution} if $i$ and $j$ are non-negative integers (resp.  integers). If $q\nmid p+m-1$, then  there does not exist $j$ such that $(0,j)$ satisfies \eqref{eq9}. Therefore $P_m(0,y)=0$, which implies that $x=0$ is an invariant lines of system \eqref{eq1}. Since Lemma \ref{l10} shows that the origin is a unique singular point of system \eqref{eq1}, system \eqref{eq1} has no periodic orbit.

If $p\nmid q+n-1$, then by the same arguments as above we conclude that  $y=0$ is is an invariant lines of system \eqref{eq1}. Therefore system \eqref{eq1} has no periodic orbit.
\end{proof}

To prove the assertion (ii) of Theorem \ref{mainth}, we give the following lemma.
 \begin{lemma}\label{l11}
Suppose that the Covention (i) holds, $q$ is even and  $n$  is odd, then system \eqref{eq1'} has no periodic orbit.
\end{lemma}
 \begin{proof}
 By direct computation we have
 \begin{equation*}
 \dfrac{dy}{dt}\Big|_{y=0}=Q_n(x,0)=b_0x^{r_2}=b_0(x^{1/p})^{q-1+n}.
 \end{equation*}
 If $b_0\not=0$ and $n$ is odd, then $(dy/dt)|_{y=0}$  does not change the sign.  Since the flow on $x$-axis have the same direction in the phase plane, there is no periodic orbit surrounding the origin.

 If $b_0=0$, then $y=0$ is an invariant line of system \eqref{eq1'}, which yields the assertion. This finishes the proof.
 \end{proof}

\begin{proof} [Proof of Theorem \ref{mainth}(ii)] If $q| p+m-1$, then $(0,r_1)$ is a non-negative integer solution of equation \eqref{eq9}. Moreover, all the integer solution of equation \eqref{eq9} are
\begin{equation*}
\{(iq,r_1-ip): i\in \mathbb{Z}\},
\end{equation*}
which implies that $P_m(x,y)=\sum_{i=0}^{[r_1/p]}a_{i}x^{iq}y^{r_1-ip}$. By the same arguments, we have $Q_n(x,y)=\sum_{j=0}^{[r_2/q]}b_jx^{r_2-jq}y^{jp}$.

Suppose that system \eqref{eq1'} has at least one periodic orbit.  Note that system \eqref{eq1'} has a unique singular point at the origin and hence all periodic orbits surround the origin.

We split the proof into two cases.

    (a) Assume that  $p$ and $q$ are odd. By applying the definition of a $(p,q)$ semi-quasi-homogeneous system, we have
\begin{equation*}
P_m(-x,-y)=P_m((-1)^px,(-1)^qy)=(-1)^{p-1+m}P_m(x,y)=(-1)^mP(x,y).
\end{equation*}
If $m$ is even, then $P_m(-x,-y)=P(x,y)$, which implies that $dx/dt|_{x=0}=P(0,y)=P(0,-y)$. Hence the flow on $x=0$ goes either from left to right, or from right to left in the phase plane, in contradiction with the assumption that system \eqref{eq1'} has at least a period orbit. Therefore $m$ is odd, which gives $P_m(-x,-y)=-P_m(x,y)$. It follows from \eqref{eq1'} that
\begin{equation*} P_m(-x,-y)=\sum_{i=0}^{[r_1/p]}a_{i}(-1)^{r_1+(q-p)i}x^{iq}y^{r_1-ip}=-\sum_{i=0}^{[r_1/p]}a_{i}x^{iq}y^{r_1-ip}=-P_m(x,y).
\end{equation*}
Since both $p$ and $q$ are odd, one obtains $a_i((-1)^{r_1}+1)=0$. If $r_1$ is even, then $a_i=0$ for $i=0,1,2,\cdots,[r_1/p]$, i.e. $P_m(x,y)\equiv 0$. In this case system \eqref{eq1'} has no periodic orbit. Therefore $r_1$ is odd.

Using the same arguments as above, we conclude that both $n$ and $r_2$ are odd. The statement (ii.1) is proved.

(b)  Assume that  $p$ is odd and  $q$ is  even. By  the definition of $r_1$ and $r_2$,  we conclude that  $m$ is even, and $r_2$ is odd (resp. even) if and only if $n$ is even (resp. odd).
Since we suppose that system \eqref{eq1'} has at least one periodic orbit, it follows from Lemma \ref{l11} that $n$ is even, which implies that $r_2$ is odd. In what follows  we only need to prove that $r_1$ is odd.

 Note that
\begin{equation*}
\frac{dx}{dt}\Big|_{x=0}=a_0y^{r_1}.
\end{equation*}
If $r_1$ is even, then $(dx/dt)|_{x=0}$ does not change the sign or equals to zero identically, which implies that system \eqref{eq1'} has no periodic orbit. The assertion (ii.2) follows.
\end{proof}

The following theorem will be used to prove the assertion (iii) of Theorem \ref{mainth}.
\begin{theorem}\label{th12}
Suppose that (ii.1)  of Theorem \ref{mainth} holds and $q| p+m-1,\,\,p| q+n-1$ , then  there exist $\bar P_m(x,y)$ and $\bar Q_n(x,y)$ such that system
\begin{equation}\label{p}
\frac{dx}{dt}=y^{r_1}+\epsilon \bar P_m(x,y),\,\,\frac{dy}{dt}=-x^{r_2}+\epsilon \bar Q_n(x,y)
\end{equation}
has at least $[([r_1/p]+1)/2]+[([r_2/q]+1)/2]-1$  limit cycles, where $\epsilon$ is a small parameter, $r_1,\,r_2$ are defined in Theorem \ref{mainth},   and
\begin{equation}
\bar P_m(x,y)=\sum_{i=1}^{[r_1/p]}a_{i}x^{iq}y^{r_1-ip},\,\,\,\,\,\bar Q_n(x,y)=\sum_{j=1}^{[r_2/q]}b_jx^{r_2-jq}y^{jp}.
\end{equation}
\end{theorem}
\begin{proof}
Let  $l_1=(r_1+1)/2,\,l_2=(r_2+1)/2$. System \eqref{p} is a perturbed system of  $(l_1,l_2)$-quasi-homogeneous Hamiltonian system
\begin{equation}\label{H}
\frac{dx}{dt}=y^{r_1},\,\,\frac{dy}{dt}=-x^{r_2}.
\end{equation}
 It is well known  (see for instance \cite{R}) that each simple zero of Abelian integral
\begin{equation*}
I(\rho)=\oint_{\Gamma_\rho}-\bar{P}_m(x,y)dy+\bar{Q}_n(x,y)dx,
\end{equation*}
produces a periodic orbit of system \eqref{p}, where $\Gamma_\rho=\{(x,y):H(x,y)=l_1x^{2l_2}+l_2y^{2l_1}=\rho^{2l_1l_2}\}$ is a closed orbit of system \eqref{H}. Taking the change \eqref{pc},
it follows from Lemma \ref{l6} and Lemma \ref{l7} that
\begin{eqnarray*}
I(\rho)&=&\sum_{i=1}^{[r_1/p]}a_i\rho^{i(l_1q-l_2p)+r_1l_2+l_2}\int_0^T\mathrm{Cs}^{iq+2l_2-1}\phi\mathrm{Sn}^{r_1-ip}\phi d\phi\\[2ex]
&&+\sum_{j=1}^{[r_2/q]}b_j\rho^{-j(l_1q-l_2p)+r_2l_1+l_1}\int_0^T\mathrm{Cs}^{r_2-jq}\mathrm{Sn}^{jp+2l_1-1}\phi d\phi\\[2ex]
&=&\sum_{i=1}^{[r_1/p]}a_i\rho^{i(m-n)/2+2l_1l_2}\int_0^T\mathrm{Cs}^{iq+2l_2-1}\phi\mathrm{Sn}^{r_1-ip}\phi d\phi\\[2ex]
&&+\sum_{j=1}^{[r_2/q]}b_j\rho^{-j(m-n)/2+2l_1l_2}\int_0^T\mathrm{Cs}^{r_2-jq}\mathrm{Sn}^{jp+2l_1-1}\phi d\phi\\[2ex]
&=&\rho^{2l_1l_2}\left(\sum_{i=1}^{[([r_1/p]+1)/2]}\mu_{2i-1}\rho^{(2i-1)(m-n)/2}+\sum_{j=1}^{[([r_2/q]+1)/2]}\nu_{2j-1}\rho^{-(2j-1)(m-n)/2}\right),\\[2ex]
\end{eqnarray*}
where
\begin{equation*}
\mu_i=a_i\int_0^T\mathrm{Cs}^{iq+2l_2-1}\phi\mathrm{Sn}^{r_1-ip}\phi d\phi,\,\,\nu_j=b_j\int_0^T\mathrm{Cs}^{r_2-jq}\mathrm{Sn}^{jp+2l_1-1}\phi d\phi.
\end{equation*}
If $m\leq n$, then
\begin{equation*}
I(\rho)=\rho^{2l_1l_2+(2[([r_1/p]+1)/2]-1)(m-n)/2}F(\rho^{n-m}),
\end{equation*}
where
\begin{equation*}
F(\xi)=\sum_{i=1}^{[([r_1/p]+1)/2]}\mu_{2i-1}\xi^{[([r_1/p]+1)/2]-i}+\sum_{j=1}^{[([r_2/q]+1)/2]}\nu_{2j-1}\xi^{j+[([r_1/p]+1)/2]-1}.
\end{equation*}
Since $F(\xi)$ is a polynomial of $\xi$ with degree at most $[([r_1/p]+1)/2]+[([r_2/q]+1)/2]-1$, $I(\rho)$ has at most $[([r_1/p]+1)/2]+[([r_2/q]+1)/2]-1$ zeros in $(0,\infty)$.
By a suitable choice of $\mu_{2i-1}$'s and $\nu_{2i-1}$'s, we can find a perturbation such that $I(\rho)$ has exactly  $[([r_1/p]+1)/2]+[([r_2/q]+1)/2]-1$ zeros in $(0,\infty)$.

If $m>n$, then result follows by the same arguments as above.
\end{proof}

\begin{proof}[Proof of Theorem \ref{mainth}(iii)] Suppose that (ii.1)  holds.  Consider the $(p,q)$ semi-quasi-homogeneous system \eqref{H}.
 Note System \eqref{H} is also a $(l_1,l_2)$-quasi-homogeneous  Hamiltonian system with a first integral $H(x,y)=l_1x^{2l_2}+l_2y^{2l_1}$. The origin is a center of system \eqref{H} and all solutions are periodic. This proves (a).

 If $b_0=0$, then $y=0$ is an invariant line of system \eqref{eq1'} and system \eqref{eq1'} has no periodic orbit in this case. This proves (b).

 Other assertions in Theorem \ref{mainth}(iii) follow from Theorem \ref{th12}.
\end{proof}

\begin{proof}[Proof of Theorem \ref{mainth}(iv)] Suppose that $L$ is a limit cycle of system \eqref{eq1'}. Let $\varphi_t(0,y_0)$ be a orbit  such that $L$ is a subset of $\omega$ (or $\alpha$) limit set of  $\varphi_t(y_0,0)$ and $\varphi_0(0,y_0)=(0,y_0)$.
By Poincar\'{e}-Bendixson theorem  $\omega(\varphi_t(0,y_0))=L$ (or  $\alpha(\varphi_t(0,y_0))=L$). This implies that there exists $t_1\in \mathbb{R}$ satisfying $\varphi_{t_1}(0,y_0)=(0,y_0')$. On the other hand, if $p$ is odd and $q$ is even, then $P_m(-x,y)=P(x,y),\,\,Q_n(-x,y)=-Q_n(x,y)$, which yields that the phase portraits of system \eqref{eq1'} are symmetric with respect $y$-axis. Therefore $\varphi_t(0,y_0)$ is a periodic orbit. This shows that $L$ is not a limit cycle ( the isolated periodic orbit). Since  there is no other singular point in the finite plane, the origin is a center of system \eqref{eq1'}.
\end{proof}

\section{Proof Theorem \ref{th3}-\ref{th5}}

In this section we will prove Theorem \ref{th3}-\ref{th5}.

\begin{proof}[Proof of Theorem \ref{th3}] Without loss of generality suppose $a_0<0,\,b_0>0$. Taking the changes
\begin{equation}\label{eq14}
\bar x=b_0^{1/(r_2+1)}x,\,\,\bar y=(-a_0)^{1/(r_1+1)}y,\,\,\tau =(-a_0)^{1/(r_1+1)}b_0^{1/(r_2+1)}t,
\end{equation}
system \eqref{eq3} becomes
\begin{equation}\label{eq15}
\frac{d\bar x}{d\tau}=-\bar y^{r_1}+\sum_{i=1}^{[r_1/p]}\bar a_i\bar x^{iq}\bar y^{r_1-ip},\,\,\,\,\,\frac{d\bar y}{d\tau}=\bar x^{r_2},
\end{equation}
where $\bar a_i=a_i(-a_0)^{ip/(r_1+1)-1}b_0^{-(iq+1)/(r_2+1)}$. Let $r_1=2l_1-1$ and $r_2=2l_2-1$. Note that $-\bar y^{r_1}\partial/\partial \bar x+\bar x^{r_2}\partial /\partial \bar y$ is a $(l_1,l_2)$-quasi-homogeneous vector filed of weighted degree
$2l_1l_2-l_2-l_2+1$, and

(i) If   $\bar a_1=\bar a_3=\cdots=\bar a_\kappa=0$, then origin is a center of system \eqref{eq14}.

 To prove this assertion, we rewrite system \eqref{eq15} in $(\rho,\phi)$ coordinates by  $(l_1,l_2)$-polar coordinates  introduced in \eqref{pc}:
 \begin{eqnarray*}
 \dfrac{d\rho}{d\tau}&=&\rho^{2l_1l_2-l_1-l_2+1}\sum_{i=1}^{[r_1/p]}\rho^{i(m-n)/2}\bar a_i\mathrm{Cs}^{iq+2l_2-1}\phi \mathrm{Sn}^{r_1-ip}\phi,\\[2ex]
 \dfrac{d\phi}{d\tau}&=&\rho^{2l_1l_2-l_1-l_2}\left(1-l_2\sum_{i=1}^{[r_1/p]}\rho^{i(m-n)/2}\bar a_i\mathrm{Cs}^{iq+2l_2-1}\phi \mathrm{Sn}^{r_1-ip+1}\phi\right),
 \end{eqnarray*}
which implies that
\begin{equation}\label{eq16}
\frac{1}{\rho^{1+(m-n)/2}}\frac{d\rho}{d\phi}=\frac{\sum_{i=1}^{[r_1/p]}\rho^{(i-1)(m-n)/2}\bar a_i\mathrm{Cs}^{iq+2l_2-1}\phi \mathrm{Sn}^{r_1-ip}\phi}{1-l_2\sum_{i=1}^{[r_1/p]}\rho^{i(m-n)/2}\bar a_i\mathrm{Cs}^{iq+2l_2-1}\phi \mathrm{Sn}^{r_1-ip+1}\phi}\triangleq \frac{A(\rho,\phi)}{G(\rho,\phi)}.
\end{equation}
Recall that $\mathrm{Cs}\phi$ and $\mathrm{Sn}\phi$ are $T$-periodic functions, where $T$ is defined in Lemma \ref{l6}. For $\phi\in [0,T]$ and small $\rho$,  we have $G(\rho,\phi)>1/2$. This yields
that
\begin{equation*}
\left|\frac{A(\rho,\phi)}{G(\rho,\phi)}\right|<M<+\infty,
\end{equation*}
for $\phi\in [0,T]$ as $\rho\rightarrow 0$, where $M$ is a non-negative constant. Integrating both side of \eqref{eq16}, one gets
\begin{equation}
|\rho_2^{-(m-n)/2}-\rho_1^{-(m-n)/2}|=\left|\int_{\phi_1}^{\phi_2}\frac{A(\rho,\phi)}{G(\rho,\phi)}d\phi\right|<M|\phi_2-\phi_1|<+\infty,
\end{equation}
where $0\leq \phi_1<\phi_2\leq T$ and $\rho(\phi_i)=\rho_i,\,\,i=1,2$. The above inequality shows that  $\rho(\phi_2)\not=0$ if $\rho(\phi_1)\not=0$. Therefore the origin is either a focus or a center.

Let $X=(\bar P_m,\bar Q_n)$ be the vector field associated system \eqref{eq15}. If $\bar a_1=\bar a_3=\cdots=\bar a_\kappa=0$, then $\bar P_m(-x,y)=\bar P(x,y),\,\,\bar Q_m(-x,y)=-\bar Q_m(x,y)$, which shows that the phase portraits of system \eqref{eq15} are symmetric with respect $y$-axis . Since the origin is either a focus or a center, the assertion in this part follows.

(ii) If $m>n$, then $-\bar x^{iq}\bar y^{r_1-ip}\partial/\partial \bar x$  is a $(l_1,l_2)$-quasi-homogeneous vector filed of weighted degree
$2l_1l_2-l_2+1+i(m-n)>2l_1l_2-l_1-l_2+1$ for $i\geq 1$. Furthermore, if $i$ is odd, then it follows from Lemma \ref{l6} that
\begin{eqnarray*}
I_i&=&\int\int_{l_1\bar x^{2l_2}+l_2\bar y^{2l_1}\leq 1}\mathrm{div}(-\bar x^{iq}\bar y^{r_1-ip},0)d\bar xd\bar y\\[2ex]
&=&-\oint_{l_1\bar x^{2l_2}+l_2\bar y^{2l_1}=1}\bar x^{iq}\bar y^{r_1-ip}d\bar y\\[2ex]
&=&-\int_0^T\mathrm{Cs}^{iq+2l_2-1}\phi\mathrm{Sn}^{r_1-ip}\phi d\phi\not=0.
\end{eqnarray*}
The results of this theorem  follow from Theorem \ref{th9}.
\end{proof}

\begin{proof}[Proof of Theorem \ref{th4}] If $m=1$, then $r_1=p/q$. Noting that convention (i) holds, we have $q=1$. On the other hand $p|q+n-1$ and $q=1$ imply that $p|n$. Therefore $r_2=n/p$ is odd integer. System \eqref{eq4} follows from \eqref{eq1'}.

(i) Suppose  $a_0=0,\,\,a_1\not=0$. Since  $(P_1(x,y),Q_n(x,y))=1$, one gets $b_{n/p}\not=0$.  The other assertion of (i) follows from Theorem \ref{th8}.

(ii) Suppose $a_0\not=0,\,\,a_1\not=0$.

Assume  $p=1$. By means of change $\bar x=-a_1x-a_0y,\,\,\bar y=-a_0 y$, system \eqref{eq4} becomes
\begin{equation}\label{eq18}
\frac{d\bar x}{dt}=a_1\bar x+a_0\sum_{l=0}^nb_j\left(\frac{\bar x-\bar y}{a_1}\right)^{n-j}\left(\frac{\bar y}{a_0}\right)^{j},\,\,\frac{d\bar y}{dt}=a_0\sum_{l=0}^{n}b_j\left(\frac{\bar x-\bar y}{a_1}\right)^{n-j}\left(\frac{\bar y}{a_0}\right)^{j}.
\end{equation}
Using Theorem \ref{th8} again, the assertion (a.1) and (a.2) follow from \eqref{eq18}.

If $p\geq 3$, then we get the assertion (b) by Theorem \ref{th8}.

(iii) The assertion (iii) follows from  Theorem \ref{th3}.

\end{proof}

\begin{proof}[Proof of Theorem \ref{th5}]
By the same arguments as in the proof of Theorem \ref{th3}, the results follow from Remark \ref{r10}.
\end{proof}

\end{document}